\documentclass{amsart}
\usepackage{amsmath,amsfonts,amsthm}
\usepackage{amssymb,latexsym}
\usepackage{graphics}
\usepackage[colorlinks]{hyperref}

\usepackage[all]{xypic}

\theoremstyle{plain}
\theoremstyle{definition}
\newtheorem{theorem}{Theorem}[section]
\newtheorem{lemma}[theorem]{Lemma}

\newtheorem{note}[theorem]{Note}

\newtheorem{convention}[theorem]{Convention}

\theoremstyle{remark}

\numberwithin{equation}{section}
\newcommand{\SP}{\: \: \: \: \:}

\title{Top--designs in the category of Fort spaces}
\author[M. Pourattar, F. Ayatollah Zadeh Shirazi]{Mehrnaz Pourattar, Fatemah Ayatollah Zadeh Shirazi}
\begin{document}
\begin{abstract} 
In infinite topological Fort space $X$, for nonempty
subsets $C,D$ of $X$ in the following text we answer to this question
``Is there any $\lambda$ and Top--design $C-(X,D,\lambda)$ of
type $i$?'' for $i=1,2,3,4$. We prove
there exist $\lambda$ and  $C-(X,D,\lambda)$, Top--design of type 2 (resp.
type 4) if and only if $C$ can be embedded into $D$.
\end{abstract}
\maketitle
\noindent {\small {\bf 2010 Mathematics Subject Classification:}  54,  05B05 \\
{\bf Keywords:}} Fort space, Generalized--design, Top--design.
\section{Introduction}
Suppose $S$ is a finite set with $n\geq2$ elements (so $S$ is an $n-$set) and $\mathbb A$ is a collection
of $k-$subsets of $S$ such that each $t-$subset of $S$ occurs exactly in $\lambda$ elements of $\mathbb A$,
then $\mathbb A$ is favorit and well studied traditional $t-(n,k,\lambda)$ combinatorial design ($t<n$ and $\lambda\geq1$)
(see \cite{999, 888}). However these finite traditional designs has been generalized in ``infinite designs'' in \cite{888}, also
generalized designs have been introduced for the first time in \cite{111} as a generalization of combinatorial designs in different
mathematical categories like category of well--ordered sets, topological spaces, etc..
We use term Top-design when our reference category is the category of topological spaces.
\\
Using the same notations as in \cite{111}, in topological space $X$ for
nonempty subsets $C,D$ of $X$, nonzero cardinal number $\lambda$
and collection $\mathbb{A}$ of subsets of $X$ using statements
(where by $S\approx T$ we mean $S$ and $T$ are homeomorphic spaces):
\begin{itemize}
\item[I.] $\forall B\in{\mathbb A}\:\:\:(B\approx D)$
\item[II.] $\forall B\in{\mathbb A}\:\:\:(B\approx D\wedge X\setminus B\approx X\setminus D)$
\item[III.] $\forall E\subseteq X\:\:\:(E\approx C\Rightarrow{\rm card}(\{B\in{\mathbb A}:E\subseteq B\})=\lambda)$
\item[IV.] $\forall E\subseteq X\:\:\:((E\approx C \wedge X\setminus E\approx X\setminus C)\Rightarrow{\rm card}(\{B\in{\mathbb A}:E\subseteq B\})=\lambda)$,
\end{itemize}
we say $\mathbb{A}$
is a :
\begin{itemize}
\item  $C-(X,D,\lambda)$ Top-design of type 1, if (II) and (III)
\item  $C-(X,D,\lambda)$ Top-design of type 2, if (I)  and (III)
\item  $C-(X,D,\lambda)$ Top-design of type 3, if (II) and (IV)
\item  $C-(X,D,\lambda)$ Top-design of type 4, if (I) and (IV)
\end{itemize}
Let's mention that if $b\in X$, equip $X$ with topology $\{U\subseteq X:b\notin U \vee
(X\setminus U$ is finite$)\}$, then we say $X$ is a Fort space with
particular point $b$ \cite[Counterexample 24]{222}. One may find counterexamples regarding
$C-(\{\frac1n:n\geq1\}\cup\{0\},D,\lambda)$ Top-designs
in \cite{111}, note to the fact that $\{\frac1n:n\geq1\}\cup\{0\}$
(with induced topology of $\mathbb R$) is an infinite countable Fort space,
leads us to study other types of infinite Fort spaces in the approach
of Top-designs.
\begin{note}
Two Fort spaces are homeomorph if and only if they are in one--to--one correspondence.
Moreover in Fort space $X$ with particular point $b$ infinite subset $Y$ of $X$
as subspace topology has Fort topology if and only if $b\in Y$ (all finite subsets
of $X$ are finite discrete spaces and carry Fort topology structure).
\end{note}
\begin{convention}
In the following text suppose $X$ is an infinite Fort
space with the particular point $b$.
\end{convention}
\section{Results in Top-designs on $X$}
In this section we study the existence of $C-(X,D,\lambda)$ for
different $C$s and $D$s.
\begin{lemma}\label{taha10}
For $U,V\subseteq X$ with $U\approx V$ and $X\setminus U\approx X\setminus V$ we have:
\begin{itemize}
\item[1.] $b\in U$ if and only if $b\in V$ (i.e., $U\cap\{b\}=V\cap\{b\}$),
\item[2.] for infinite $U$ with ${\rm card}(U)<{\rm card}(X)$
	and $H\subseteq X$ we have $U\approx H$ and $X\setminus U\approx X\setminus H$
	if and only if ${\rm card}(U)={\rm card}(H)$ and $U\cap\{b\}=H\cap\{b\}$.
\end{itemize}
\end{lemma}
\begin{proof}
1) First suppose $U$ is infinite, so $V$ is infinite too. Since $b$ is the unique limit point of any infinite subset of $X$,
$U$ contains a limit point if and only if $b\in U$ on the other hand $U$ contains a limit point
if and only if $V$ contains a limit point which means $b\in V$ in its turn.
\\
Now suppose $U$ is finite, thus $X\setminus U$ is infinite and using a similar method described above,
we have $b\in X\setminus U$ if and only if $b\in X\setminus V$ which completes the proof.
\\
2) Suppose ${\rm card}(U)={\rm card}(H)$ and $U\cap\{b\}=H\cap\{b\}$, then
${\rm card}(U\setminus\{b\})={\rm card}(H\setminus\{b\})$,
thus there exists bijection $f:U\setminus\{b\}\to H\setminus\{b\}$. If $b\notin U$,
then $b\notin H$ and $f:U\setminus\{b\}=U\to H\setminus\{b\}=H$
is a homeomorphism (of discrete spaces) too. If $b\in U$, then $b\in H$ too
and $\tilde{f}:U\to H$ with $\tilde{f}\restriction_{U\setminus\{b\}}=f$
and $\tilde{f}(b)=b$ is a homeomorphism of infinite Fort spaces
($U$ and $H$ with particular point $b$). So $U\approx H$.
\\
On the other hand if ${\rm card}(U)={\rm card}(H)<{\rm card}(X)$, then
${\rm card}(X\setminus U)={\rm card}(X\setminus H)={\rm card}(X)$. Also if
 $U\cap\{b\}=H\cap\{b\}$, then  $(X\setminus U)\cap\{b\}=(X\setminus H)\cap\{b\}$. So if
 ${\rm card}(U)={\rm card}(H)<{\rm card}(X)$ and $U\cap\{b\}=H\cap\{b\}$, then
 ${\rm card}(X\setminus U)={\rm card}(X\setminus H)$ and
 $(X\setminus U)\cap\{b\}=(X\setminus H)\cap\{b\}$ which shows
$X\setminus U\approx X\setminus H$ by the above argument.
\\
Use item (1) to complete the proof of (2).
\end{proof}
Note that if there exists a $C-(X,D,\lambda)$, Top--design of type $i$, then there exists
$U\approx D$ with $C\subseteq U$, so ${\rm card}(C)\leq {\rm card}(U)={\rm card}(D)$.
Therefore ${\rm card}(C)=\min({\rm card}(D),{\rm card}(C))\leq{\rm card}(X)$.
\begin{theorem}\label{salam10}
Regarding 1st type of Top--designs for
nonempty subsets $C,D$ of $X$ we have:
\\
a. suppose $b\notin C\cup D$:
	\begin{itemize}
	\item[a$_1$.] if $C$ is finite, then there is not any
		$C-(X,D,\lambda)$, Top--design of type 1,
	\item[a$_2$.]  if  $C$ is infinite
		and ${\rm card}(C)<{\rm card}(X)$,
		then there exist $\lambda$ and a $C-(X,D,\lambda)$
		Top--design of type 1,
	\item[a$_3$.] if
		${\rm card}(C)={\rm card}(D)={\rm card}(X)$,
		then there exist $\lambda$ and a $C-(X,D,\lambda)$
		Top--design of type 1 if and only if $D=X\setminus\{b\}$,
	\end{itemize}
b. if $b\in C\setminus D$, then there is not any $C-(X,D,\lambda)$, Top--design of type 1,
\\
c. suppose $b\in D$:
	\begin{itemize}
	\item[c$_1$.] for finite $C$ there exist $\lambda$ and a $C-(X,D,\lambda)$
		Top--design of type 1 if and only if ${\rm card}(C)+2\leq{\rm card}(D)$,
	\item[c$_2$.] if  $C$ is infinite
		and ${\rm card}(C)=\min({\rm card}(D),{\rm card}(C))<{\rm card}(X)$,
		then there exist $\lambda$ and a $C-(X,D,\lambda)$
		Top--design of type 1,
	\item[c$_3$.] if
		${\rm card}(C)={\rm card}(D)={\rm card}(X)$,
		then there exist $\lambda$ and a $C-(X,D,\lambda)$
		Top--design of type 1 if and only if $D=X$.
	\end{itemize}
\end{theorem}
\begin{proof} Let $\mathbb{W}=\{E\subseteq X:E\approx D\wedge X\setminus E\approx X\setminus D\}$.  By item (1) in Lemma~\ref{taha10},
it's evident that $b\in D$ if and only if
$b\in\bigcup\mathbb{W}$ (resp.
$b\in\bigcap\mathbb{W}$).
\\
$\underline{a_1}$) Choose $k\in C$, if $\mathbb A$ is a $C-(X,D,\lambda)$, Top--design of type 1, then ${\mathbb A}\subseteq{\mathbb W}$ and $\mathbb A$ is a
$(C\setminus\{k\})\cup\{b\}-(X,D,\lambda)$, Top--design of type 1 too,
which is a contradiction since $b\notin \bigcup{\mathbb W}$.
\\
$\underline{a_2}$) We have the following sub--cases:
\\
$\bullet$ ${\rm card}(C)\leq{\rm card}(D)<{\rm card}(X)$. In this case by item (2) in Lemma~\ref{taha10} we have
	${\mathbb W}=\{E\subseteq X\setminus\{b\}:
	{\rm card}(E)={\rm card}(D)\}$. Using ${\rm card}(X\setminus\{b\})={\rm card}
	(X\setminus(C\cup\{b\}))$ for ${\mathbb W}'=\{E\subseteq
	X\setminus(C\cup\{b\}): {\rm card}(E)={\rm card}(D)\}$ we have
	${\rm card}({\mathbb W})={\rm card}({\mathbb W}')$. It's evident
	that $\mathop{{\mathbb W}'\to{\mathbb W}}\limits_{E\mapsto E\cup C}$
	is one--to--one, so ${\rm card}({\mathbb W}')\leq
	{\rm card}(\{F\in{\mathbb W}:C\subseteq F\})\leq {\rm card}({\mathbb W})$.
	Thus ${\rm card}(\{F\in{\mathbb W}:C\subseteq F\})
	= {\rm card}({\mathbb W})$.
	Since $C$  is infinite and $b\notin C$, for all subset $E$ of $X$ with
	$C\approx E$ we have ${\rm card}(C)={\rm card}(E)$ and $b\notin E$,
	so by a similar method described for $C$ we have
	${\rm card}(\{F\in{\mathbb W}:E\subseteq F\})= {\rm card}({\mathbb W})$.
	Hence $\mathbb W$ is a $C-(X,D,{\rm card}({\mathbb W}))$ Top--design
	of type 1.
\\
$\bullet$ ${\rm card}(C)<{\rm card}(D)={\rm card}(X)$.
	In this case by Lemma~\ref{taha10}, ${\mathbb W}=\{E\subseteq X\setminus\{b\}:
	{\rm card}(E)={\rm card}(D)\wedge
	{\rm card}(X\setminus E)={\rm card}(X\setminus D)\}$. Since ${\rm card}(C)<{\rm card}(D)$
	and $C,D$ carry discrete topologies thus $C$ can be embedded in $D$ and without any loss of generality we may suppose
	$C\subseteq D$. By infiniteness of $D$, at least one of the sets $D\setminus C$ or $C$ is infinite
	and
	\begin{eqnarray*}
	{\rm card}(C) & < & {\rm card}(X)={\rm card}(D) \\
	& = & {\rm card}(C)+{\rm card}(D\setminus C) \\
	& = & \max({\rm card}(C),{\rm card}(D\setminus C))
	\end{eqnarray*}
	so we have $\max({\rm card}(C),{\rm card}(D\setminus C))={\rm card}(D\setminus C)={\rm card}(X)$.
	Since $2{\rm card}(D\setminus C)={\rm card}(D\setminus C)$, we may choose
	$H\subseteq D\setminus C$ with
	\[{\rm card}(H)={\rm card}
	(D\setminus C)\setminus H(={\rm card}(D\setminus (C\cup H))={\rm card}(D\setminus C)={\rm card}(X))\:.\]
Let $K=\{F\subseteq D\setminus (H\cup
	C): {\rm card}(F\cup\{b\})={\rm card}(X\setminus D)\}$, and consider the following claim:
\\
{\it Claim.} For $F\in K$ we have $C\subseteq X\setminus(F\cup\{b\})\in\mathbb{W}$. Suppose $F\in K$, so
$F\subseteq D\setminus (H\cup C)$ so $H\subseteq X\setminus(F\cup\{b\})\subseteq X$
thus ${\rm card}(X\setminus(F\cup\{b\}))={\rm card}(X)={\rm card}(D)$ and
${\rm card}(X\setminus D)={\rm card}(F\cup\{b\})={\rm card}(X\setminus(X\setminus(F\cup\{b\})))$, therefore
$X\setminus(F\cup\{b\})\in\mathbb{W}$. Also  $F\subseteq D\setminus (H\cup
	C)$ and $b\notin C$ show $C\subseteq X\setminus(F\cup\{b\})$.
\\
Therefore
	\[\eta:\mathop{K\to\{B\in{\mathbb W}:C\subseteq B\}}\limits_{F\mapsto
	X\setminus(F\cup\{b\})\:\:\:\:\:\:\:\:\:\:\:\:\:\:\:\:\:\:\:\:\:\:\:\:}\]
	is well--defined and clearly one--to--one.
\\
Thus ${\rm card}(K)\leq{\rm card}(\{B\in{\mathbb
	W}:C\subseteq B\})\leq{\rm card}({\mathbb W})$,
	however using ${\rm card}(D\setminus (H\cup C))={\rm card}(
	X\setminus\{b\})$ we have:
	\begin{eqnarray*}
	{\rm card}({\mathbb W}) & \leq & {\rm card}(\{E\subseteq X\setminus\{b\}:
	{\rm card}(E) =  {\rm card}(X\setminus (D\cup \{b\}))\}) \\
	& = & {\rm card}(\{F\subseteq D\setminus (H\cup C):
		{\rm card}(F)={\rm card}(X\setminus (D\cup\{b\}))\}) \\
	& = & {\rm card}(K)
	\end{eqnarray*}
	which leads to ${\rm card}(K)={\rm card}(\{B\in{\mathbb W}:C\subseteq
	B\})={\rm card}({\mathbb W})$.
	\\
	For $E\subseteq X$ with $E\approx C$ (so $b\notin E$), we have $D'=(D\setminus C)\cup E
	\in {\mathbb W}$ and ${\mathbb W}=\{E\subseteq X\setminus\{b\}:
	{\rm card}(E)={\rm card}(D')\wedge
	{\rm card}(X\setminus E)={\rm card}(X\setminus D')\}$.
	Using a similar method described above, we have
	${\rm card}(\{B\in{\mathbb W}:E\subseteq B\})={\rm card}({\mathbb W})$,
	thus $\mathbb W$ is a $C-(X,D,{\rm card}({\mathbb W}))$ Top--design
	of type 1.
\\
$\underline{a_3}$) In this case if ${\mathbb A}$ is a $C-(X,D,\lambda)$ Top--design
of type 1, then there exists $B\in{\mathbb A}$ with $X\setminus\{b\}\subseteq B$ (since $C\approx X\setminus\{b\}$), which leads to $D=X\setminus\{b\}$,
and ${\mathbb W}=\{X\setminus\{b\}\}$ is a
$C-(X,D,1)$ Top--design
of type 1.
\\
$\underline{b}$) Use the fact that if $b\notin D$, then for all $B\subseteq X$ with
$D\approx B$ and $X\setminus D\approx X\setminus B$ we have $b\notin B$,
and in particular $C\nsubseteq B$.
\\
$\underline{c_1}$)
First suppose $\mathbb A$ is a
	$C-(X,D,\lambda)$ Top--design of type 1 and $D$ is finite,
	then for all subsets $H$ of $X$ with ${\rm card}(H)={\rm card}(C)$,
	there exists $B\in{\mathbb A}$ with $H\subseteq B$, however we
	may assume $b\notin H$, using $b\in B$ we have ${\rm card}(H)\leq
	{\rm card}(B\setminus\{b\})={\rm card}(D)-1$. Hence
	${\rm card}(C)+1\leq{\rm card}(D)$. If ${\rm card}(C)+1={\rm card}(D)$
	then any subset of $X\setminus\{b\}$ with ${\rm card}(C)$ elements
	occurs in just one element of $\mathbb A$ and ${\mathbb A}=\{S\cup
	\{b\}:S\subseteq X\setminus\{b\}\wedge{\rm card}(S)={\rm card}(C)\}$
	now choose a subset $J$ of $X\setminus\{b\}$ with ${\rm card}(C)-1$
	elements, then infinite elements of $\mathbb A$ contain $J\cup\{b\}
	(\approx C)$ which is in contradiction with $\lambda=1$, so
	${\rm card}(C)+1<{\rm card}(D)$ and ${\rm card}(C)+2\leq{\rm card}(D)$.
\\
In order to complete the proof, we have the following cases:
\\
{\it Case 1.} $X$ is uncountable and $D$ is infinite. In this case choose infinite countable subset
	$I$ of $D\setminus\{b\}$.
By the proof of (a$_2$) for
\[\mathbb{W}_{-b}=\{E\subseteq X:E\approx D\setminus\{b\}\wedge X\setminus E\approx X\setminus (D\setminus\{b\})\}\]
	is a
$I-(X,D\setminus\{b\},{\rm card}({\mathbb W}_{-b}))$
Top--design of first type. We show $\mathbb W$ is a
$C-(X,D,{\rm card}({\mathbb W}))$
Top--design of first type. Consider $H\subseteq X$ with $H\approx C$.
There exists $J\subseteq X\setminus\{b\}$ with $H\setminus\{b\}\subseteq J$ and $J\approx I$ so
\begin{eqnarray*}
{\rm card}({\mathbb W}_{-b}) & \geq & {\rm card}(\{B\in{\mathbb W}_{-b}:H\setminus\{b\}\subseteq B\}) \\
& \geq & {\rm card}(\{B\in{\mathbb W}_{-b}:J\subseteq B\})={\rm card}({\mathbb W}_{-b})
\end{eqnarray*}
therefore ${\rm card}(\{B\in{\mathbb W}_{-b}:H\setminus\{b\}\subseteq B\})={\rm card}({\mathbb W}_{-b})$. Considering
bijection $\eta:\mathop{{\mathbb W}_{-b}\to{\mathbb W}}\limits_{B\mapsto B\cup\{b\}}$, and $b\in\bigcap {\mathbb W}$
we have
${\rm card}(\{B\in{\mathbb W}:H\setminus\{b\}\subseteq B\})={\rm card}(\{B\in{\mathbb W}_{-b}:H\setminus\{b\}\subseteq B\})={\rm card}({\mathbb W}_{-b})={\rm card}({\mathbb W})$
which leads to
${\rm card}(\{B\in{\mathbb W}:H\subseteq B\})={\rm card}({\mathbb W})$ and
$\mathbb W$ is a
$C-(X,D,{\rm card}({\mathbb W}))$ Top--design of first type.
\\
{\it Case 2.} $X$, $D$ and $X\setminus D$   are infinite countable. In this case we may
	suppose $X\setminus\{b\}=\{p_n:n\geq1\}$ and
	$D=\{p_{2n}:n\geq1\}\cup\{b\}$ with distinct $p_n$s.
	Let ${\mathbb A}=\{X\setminus\{p_{2k+1}:k\geq s\}:s\geq1\}$,
	then ${\mathbb A}$ is a $C-(X,D,\aleph_0)$ Top--design of type 1,
\\
{\it Case 3.} $X$ and $D$ are infinite countable and $X\setminus D\neq\varnothing$ is finite. In this case
	$\mathbb W$ is infinite countable and a $C-(X,D,\aleph_0)$ Top--design of type 1.
\\
{\it Case 4.} $X=D$ is infinite countable. In this case
	${\mathbb W}=\{X\}$ is a $C-(X,D,1)$ Top--design of type 1.
\\
{\it Case 5.} $D$ is finite and ${\rm card}(C)+2\leq{\rm card}(D)$.
	In this case ${\rm card}({\mathbb W})={\rm card}(X)$
	(since for infinite set $X$ we have ${\rm card}(X)={\rm card}({\mathcal P}_{fin}(X))$,
	where ${\mathcal P}_{fin}(X))$ is the collection of all finite subsets of $X$)
	and
	${\mathbb W}$ is a
	$C-(X,D,{\rm card}(X))$ Top--design of type 1.
\\
$\underline{c_2}$) In this case by the proof of (a$_2$), $\mathbb W$ is a $C\setminus\{b\}-(X,D\setminus\{b\},{\rm card}({\mathbb W}))$ Top-design of type 1, using $b\in\bigcap{\mathbb W}$, shows that $\mathbb W$ is a $C-(X,D,{\rm card}({\mathbb W}))$ Top-design of type 1 too.
\\
$\underline{c_3}$) Use a similar method described in the proof of (a$_3$).
\end{proof}
\begin{lemma}\label{taha20}
For
nonempty subsets $C,D$ of $X$,  $C$ can be embedded into $D$ if and only if
\begin{center}
``$C$ is finite or $b\notin C\setminus D$'', and
``${\rm card}(C)\leq{\rm card}(D)$''.
\end{center}
\end{lemma}
\begin{proof}
Suppose $C$ can be embedded in $D$ and choose $E\subseteq D$ with $E\approx C$, so
${\rm card}(C)={\rm card}(E)\leq{\rm card}(D)$. If $C$ is infinite and $b\in C$ then any
subset of $X$ homeomorphic with $C$ contains $b$, thus $b\in E(\subseteq D)$ and $b\notin C\setminus D$.
\end{proof}
\begin{theorem}\label{salam20}
For
nonempty subsets $C,D$ of $X$, there exist $\lambda$ and a $C-(X,D,\lambda)$, Top--design of type 2 if and only if
$C$ can be embedded into $D$.
\end{theorem}
\begin{proof}
If we can not embed $C$
into $D$ it's evident that there is not any $C-(X,D,\lambda)$,
Top--design of type 2.
\\
Conversely suppose $C$ can be embedded in $D$, so by Lemma~\ref{taha20} ${\rm card}(C)\leq{\rm card}(D)$ and ``$C$ is finite
or $b\notin C\setminus D$''. Let ${\mathbb L}=\{E\subseteq X:E\approx D\}$.
We have the following cases:
\\
$\bullet$ ${\rm card}(C)\leq{\rm card}(D)$ and $C$ is finite. In this case
	$\mathbb L$ is a $C-(X,D,\lambda)$ Top--design of type 2 with:
	\[\lambda=\left\{\begin{array}{lc} 1 & {\rm card}(C)={\rm card}(D)\:, \\
	{\rm card}(\{E\subseteq X:{\rm card}(E)={\rm card}(D)\}) & {\rm otherwise}
	\: . \end{array}\right.\]
For this aim use the fact that $\eta:\{E\subseteq X\setminus C:{\rm card}(E)={\rm card}(D)\}\to\{E\subseteq X:{\rm card}(E)={\rm card}(D)\}$
with $\eta(E)=E\cup C$ is bijective.
\\
$\bullet$ ${\rm card}(C)=\min({\rm card}(C),{\rm card}(D))<{\rm card}(X)$ and $b\notin C\setminus D$. In this case by Theorem~\ref{salam10} there exists $\lambda$ and  $C-(X,D,\lambda)$,
Top--design of type 1, so it is a $C-(X,D,\lambda)$,
Top--design of type 2 too.
\\
$\bullet$ ${\rm card}(C)=\min({\rm card}(C),{\rm card}(D))={\rm card}(X)$ and $b\notin C\setminus D$. In this case
$\mathbb{A}=\{(X\setminus\{b\})\cup(D\cap \{b\})\}$ is a $C-(X,D,1)$ Top--design of type 2.
\end{proof}
\begin{theorem}\label{salam30}
Regarding 3rd type of Top--designs for
nonempty subsets $C,D$ of $X$, there exist $\lambda$ and a $C-(X,D,\lambda)$, Top--design of type 3 if and only if $b\notin C\setminus D$,
${\rm card}(C\setminus\{b\})\leq{\rm card}(D\setminus\{b\})$ and
${\rm card}(X\setminus(D\cup\{b\}))\leq{\rm card}(X\setminus(C\cup\{b\}))$.
\end{theorem}
\begin{proof}
Let $\mathbb{W}=\{E\subseteq X:E\approx D\wedge X\setminus E\approx X\setminus D\}$. If
$\mathbb A$ is a $C-(X,D,\lambda)$, Top--design of type 3, then ${\mathbb A}\subseteq{\mathbb W}$
and we have the following cases:
\\
{\it Case 1.} $b\in C\setminus D$. In this case for all $E\in{\mathbb A}(\subseteq{\mathbb W})$,
we have $b\notin E$ and $C\not\subseteq E$
thus $\mathbb A$ is not a  $C-(X,D,\lambda)$, Top--design of type 3.
\\
{\it Case 2.} ${\rm card}(C\setminus\{b\})>{\rm card}(D\setminus\{b\})$, and ``$b\in C\cap D$ or $b\notin C\cup D$''.
In this case we have
${\rm card}(C)>{\rm card}(D)$ so
we can not embed $C$ into $D$ and it's evident that there is not any $C-(X,D,\lambda)$,
Top--design of type 3.
\\
{\it Case 3.} ${\rm card}(C\setminus\{b\})>{\rm card}(D\setminus\{b\})$, $b\in D\setminus C$.
In this case for all $B\in{\mathbb A}$, $b\in B$ and ${\rm card}(C)={\rm card}(C\setminus\{b\})>{\rm card}(D\setminus\{b\})
={\rm card}(B\setminus\{b\})$ so $C\not\subseteq B\setminus\{b\}$ and $C\not\subseteq B$ so $\mathbb A$
is not a  $C-(X,D,\lambda)$, Top--design of type 3.
\\
{\it Case 4.} ${\rm card}(X\setminus(D\cup\{b\}))>{\rm card}(X\setminus(C\cup\{b\}))$. In this case for all
$E\in{\mathbb W}$ we have ${\rm card}(X\setminus(E\cup\{b\}))>{\rm card}(X\setminus( C\cup\{b\}))$, thus $X\setminus(E\cup\{b\})\not\subseteq X\setminus( C\cup\{b\})$ and $C\not\subseteq E$, so  there is not any $C-(X,D,\lambda)$,
Top--design of type 3.
\\
Considering the above cases $b\notin C\setminus D$,
${\rm card}(C\setminus\{b\})\leq{\rm card}(D\setminus\{b\})$ and
${\rm card}(X\setminus(D\cup\{b\}))\leq{\rm card}(X\setminus(C\cup\{b\}))$.
\\
Conversely, suppose $b\notin C\setminus D$,
${\rm card}(C\setminus\{b\})\leq{\rm card}(D\setminus\{b\})$ and
${\rm card}(X\setminus(D\cup\{b\}))\leq{\rm card}(X\setminus(C\cup\{b\}))$, then $\mathbb W$
is a $C-(X,D,\lambda)$, Top--design of type 3 for
$\lambda={\rm card}(\{E\in{\mathbb W}:C\subseteq E\})$
(note that for $F\subseteq X$ with $F\approx C$ and $X\setminus F\approx X
\setminus C$, the map
$\mathop{\{E\in{\mathbb W}:F\subseteq E\}
\to \{E\in{\mathbb W}:C\subseteq E\}}\limits_{\:\:\: \:\:\:  \:\:\:\:\:\: E\mapsto
(E\setminus F)\cup C}$ is bijective).
\end{proof}
\begin{theorem}\label{taha30}
For
nonempty subsets $C,D$ of $X$, there exist $\lambda$ and a $C-(X,D,\lambda)$, Top--design of type 4 if and only if
$C$ can be embedded into $D$.
\end{theorem}
\begin{proof}
If $C$ can be embedded into $D$, then there exist $\lambda>0$ and a $C-(X,D,\lambda)$
Top--design of type 2 like $\mathbb A$ by Theorem\ref{salam20}, so $\mathbb A$ is
a $C-(X,D,\lambda)$Top--design of type 4 too.
\\
Conversely, it's evident that if $\mathbb A$ is a Top--design of type $i$
(for $i=1,2,3,4$), then there exists $E\in \mathbb{A}$ with $C\subseteq E$,
using $E\approx D$ leads us to the fact that $C$can be embedded into $D$.
\end{proof}
\begin{theorem}\label{salam40}
For
nonempty subsets $C,D$ of $X$ the following statements are equivalent:
\begin{itemize}
\item there is not any $C-(X,D,\lambda)$, Top--design of type 2,
\item there is not any $C-(X,D,\lambda)$, Top--design of type 4,
\item ``$C$ is infinite and $b\in C\setminus D$'', or
``${\rm card}(C)>{\rm card}(D)$'',
\item $C$ can not be embedded into $D$.
\end{itemize}
\end{theorem}
\begin{proof}
Theorems~\ref{salam20}, \ref{taha30} and Lemma~\ref{taha20}.
\end{proof}
\section*{Acknowledgement}
The authors are grateful to the research division of the University of Tehran
for the grant which supported this research.

\[\underline{\SP\SP\SP\SP\SP\SP\SP\SP\SP\SP\SP\SP\SP\SP\SP\SP}\]
\noindent {\small 
{\bf Mehrnaz Pourattar},
Islamic Azad University, Science and Research Branch, 
Tehran, Iran
({\it e-mail}: mpourattar@yahoo.com)
\\
{\bf Fatemah Ayatollah Zadeh Shirazi},
Faculty of Mathematics, Statistics and Computer Science,
College of Science, University of Tehran ,
Enghelab Ave., Tehran, Iran \linebreak
({\it e-mail}: fatemah@khayam.ut.ac.ir)}

\end{document}